\documentclass[a4paper]{article}

\usepackage{amsmath, amsthm, amssymb, mathrsfs}
\usepackage{graphicx, tikz}
\usepackage{hyperref}

\newtheorem{theorem}{Theorem}[section]
\newtheorem{definition}{Definition}[section]
\newtheorem{problem}{Problem}[section]
\newtheorem{lemma}{Lemma}[section]
\newtheorem{remark}{Remark}[section]

\begin{document}

\title{Uniqueness in inverse acoustic scattering with phaseless near-field measurements}

\author{
Deyue Zhang\thanks{School of Mathematics, Jilin University, Changchun, China. {\it dyzhang@jlu.edu.cn}}, 
Fenglin Sun\thanks{School of Mathematics, Jilin University, Changchun, China. {\it sunfl18@mails.jlu.edu.cn}}, 
Yukun Guo\thanks{School of Mathematics, Harbin Institute of Technology, Harbin, China. {\it ykguo@hit.edu.cn} (Corresponding author)}\ \ and 
Hongyu Liu\thanks{Department of Mathematics, Hong Kong Baptist University, Hong Kong, China. {\it hongyuliu@hkbu.edu.hk}}
}
\date{}

\maketitle

%\address{$^1$ School of Mathematics, Jilin University, Changchun, P. R. China}
%\ead{\mailto{dyzhang@jlu.edu.cn}}
%\address{$^2$ School of Mathematics, Jilin University, Changchun, P. R. China}
%\ead{\mailto{sunfl18@mails.jlu.edu.cn}}
%\address{$^3$ School of Mathematics, Harbin Institute of Technology, Harbin, P. R. China}
%\ead{\mailto{ykguo@hit.edu.cn}}
%\address{$^4$ Department of Mathematics, Hong Kong Baptist University, Hong Kong, P. R. China}
%\ead{\mailto{hongyuliu@hkbu.edu.hk}}

\begin{abstract}
This paper is devoted to the uniqueness of inverse acoustic scattering problems with the modulus of near-field data. By utilizing the superpositions of point sources as the incident waves, we rigorously prove that the phaseless near-fields collected on an admissible surface can uniquely determine the location and shape of the obstacle as well as its boundary condition and the refractive index of a medium inclusion, respectively. We also establish the uniqueness in determining a locally rough surface from the phaseless near-field data due to superpositions of point sources. These are novel uniqueness results in inverse scattering with phaseless near-field data.
\end{abstract}


%=============================================================

\section{Introduction}

Inverse scattering theory is concerned with the determination of underlying target scatterer from the incident wave and the measured near-field or far-field data. In particular, the inverse scattering theory of time-harmonic waves is of great importance in various applications such as radar detection, sonar inspection, nondestructive testing and modern medical diagnostics. The time-harmonic inverse scattering problems are typically based on complex-valued data. Hence, in terms of the accessibility to the corresponding phase information, the measured data in inverse scattering problems can be classified into two types: phased/full data and phaseless or intensity-only/modulus-only data. Over the past several decades, the inverse scattering problems with full measured data (both phase and intensity) have been mathematically and numerically studied intensively in the literature (see, e.g. \cite{CK13, KG08} and the references therein). Recently, a great deal of effort has been devoted to phaseless inverse scattering problems \cite{ACZ16, BLL13, BZ16, Kli14, Kli17, KR16, KR17}. The motivation for investigating phaseless inverse problems is mainly due to the fact that such phase information is extremely difficult to be measured accurately or even completely unavailable in a rich variety of realistic scenarios. As a result, only the phaseless data can be practically obtained in these cases. 

The inverse scattering problem with one incident plane wave and phaseless far-field data is challenging due to the {\it translation invariance property}, namely, the modulus of the far field pattern is invariant under translations \cite{KR97, LS04}. Specifically speaking, the location of the scatterer cannot be uniquely determined by the phaseless far-field data. Nevertheless, shape recovery from phaseless data is still possible. Actually, quite a number of reversion schemes have been proposed to reconstruct the shape of the scatterer from the modulus-only far-field data with a single incident plane wave, see \cite{Iva07, Iva10, Iva11, Lee16, LL15, LLW17}. We also refer to \cite{CH17, GDM18, ZZ18} for the relevant numerical studies. It was established in \cite{LZ09} that the radius of a small sound-soft ball can be uniquely determined from a single intensity-only far-field datum.

It is often desirable to develop corresponding techniques to tackle the difficulty of translation invariance. An effective attempt in this direction is the superposition of distinct incident plane waves proposed in \cite{ZZ17a}. This idea leads to the multi-frequency Newton iteration algorithm \cite{ZZ17a, ZZ17b} and the fast imaging algorithm at a fixed frequency \cite{ZZ18}. Further, by the superposition of two incident plane waves, uniqueness results were established in \cite{XZZ18a} under some a priori assumptions.

Recently, the reference ball technique was introduced in \cite{ZG18} to break the translation invariance in phaseless inverse acoustic scattering problem. By incorporating an suitably chosen ball into the scattering system as well as the superposition of incident plane wave and point sources, the authors in \cite{ZG18} rigorously prove that the location and shape of the obstacle as well as its boundary condition or the refractive index can be uniquely determined by the modulus of far-field patterns. We would like to point out that the idea of adding a reference ball to the scattering system was first proposed in \cite{LLZ09} to numerically enhance the resolution of the linear sampling method. The reference ball technique was used in \cite{XZZ18b} to alleviate the requirement of the a priori assumptions in \cite{XZZ18a}. Similar strategies of adding reference objects or sources to the scattering system have also been extensively applied to the theoretical analysis and numerical approaches for different models of phaseless inverse scattering problems \cite{DZG19, DLL18, JL18, JLZ18a, JLZ18b, ZGLL18}. In the absence of any additional reference object, the uniqueness can be established by the superposition of incident point sources and phaseless far-field data, see \cite{SZG18}.

In this paper, we will deal with the uniqueness issue concerning the inverse acoustic scattering problems with incident point sources and phaseless near-field data.  In the areas of optics and engineering sciences, the phaseless inverse scattering with near-field data is also known as phase retrieval problem \cite{Maleki, Maleki1}. The inverse scattering problems with phaseless near-field data have been studied numerically (see, e.g. \cite{Candes, Candes1, Caorsi, CH17, CFH17, Pan, Takenaka}), and few studies have been made on the theoretical aspects of uniqueness for the inverse scattering problems. A recent result on uniqueness in \cite{Kli14} was related to the reconstruction of a potential with the phaseless near-field data for point sources on a spherical surface and an interval of wavenumbers, which was extended in \cite{Kli17} to determine the wave speed in generalized 3-D Helmholtz equation. The uniqueness of a coefficient inverse scattering problem with phaseless near-field data has been established in \cite{KR17}. We also refer to \cite{KR16, Novikov15, Novikov16} for some recovery algorithms for the inverse medium scattering problems with phaseless near-field data. The stability analysis for linearized near-field phase retrieval in X-ray phase contrast imaging can be found in \cite{Maretzke}.

In this work, we establish the uniqueness via superposition of incident point sources, which does not rely on any additional reference/interfering scatterer. By introducing the concept of an admissible surface or curve, together with the superposition of point sources, we rigorously proved that the bounded scatterer (impenetrable obstacle or medium inclusion) and the locally perturbed half-plane (a.k.a locally rough surface) could be uniquely determined from the phaseless near-field measurements. For the uniqueness of inverse scattering by locally rough surfaces with phaseless far-field data, we refer to \cite{XZZ18b}. A key feature of this study is that we make use of the limited-aperture phaseless near-field data co-produced by the scatterer and point sources, thus the configuration is practically more feasible than the cases of using the phaseless scattered data.

The rest of this paper is arranged as follows. Section \ref{sec:bounded_scatterer} is devoted to the inverse scattering problem of uniquely determining a bounded scatterer. Then in section \ref{sec:half-plane}, we study the uniqueness results on phaseless inverse scattering by locally perturbed half-planes.

%%%%%%%%%%%%%%%%%%%%%%%%%%%%%%%%%%%%%%%%%%%%%%%%%%%%%%%%
\section{Uniqueness for inverse scattering by bounded scatterers}\label{sec:bounded_scatterer}
%%%%%%%%%%%%%%%%%%%%%%%%%%%%%%%%%%%%%%%%%%%%%%%%%%%%%%%%

\subsection{Problem setting}

We begin this section with the acoustic scattering problems for an incident plane wave. Assume $D \subset\mathbb{R}^3$ is an open and simply-connected domain with $C^2$ boundary $\partial D$.  Denote by $\nu$ the unit outward normal to $\partial D$ and by $\mathbb{S}^2:=\{x\in\mathbb{R}^3: |x|=1\}$ the unit sphere in $\mathbb{R}^3$. Let $u^i(x,d)=\mathrm{e}^{\mathrm{i} k x\cdot d}$ be a given incident plane wave, where $d\in\mathbb{S}^2$ and $k>0$ are the incident direction and wavenumber, respectively.
Then, the obstacle scattering problem can be formulated as: to find the total field $u=u^i+u^s$ which satisfies the following boundary value problem (see \cite{CK13}):
\begin{eqnarray}
\Delta u+ k^2 u  =  0\quad \mathrm{in}\ \mathbb{R}^3\backslash\overline{D},\label{eq:Helmholtz} \\
\mathscr{B}u  =  0 \quad \mathrm{on}\ \partial D, \label{eq:boundary_condition} \\
\lim\limits_{r=|x|\rightarrow\infty} r\bigg(\displaystyle\frac{\partial u^s}{\partial r} -\mathrm{i} ku^s\bigg)=0, \label{eq:Sommerfeld}
\end{eqnarray}
where $u^s$ denotes the scattered field and \eqref{eq:Sommerfeld} is the Sommerfeld radiation condition. Here $\mathscr{B}$ in \eqref{eq:boundary_condition} is the boundary operator
defined by
\begin{eqnarray}\label{BC}
\left\{
\begin{array}{ll}
\mathscr{B}u=u & \text{for a sound-soft obstacle},  \\
\mathscr{B}u=\displaystyle\frac{\partial u}{\partial \nu}+\mathrm{i} k\lambda u & \text{for an impedance obstacle},
\end{array}
\right.
\end{eqnarray}
where $\lambda$ is a real parameter. This boundary condition \eqref{BC} covers the Dirichlet/sound-soft boundary condition, the Neumann/sound-hard boundary condition ($\lambda=0$), and the impedance boundary condition ($\lambda\neq 0$).

The medium scattering problem is to find the total field $u=u^i+u^s$ that fulfills
\begin{eqnarray}
    \Delta u+ k^2n(x) u=  0 \quad \text{in}\ \mathbb{R}^3, \label{eq:Helmholtz_D}\\
    \lim\limits_{r=|x|\to\infty} r\bigg(\displaystyle\frac{\partial u^s}{\partial r}-\mathrm{i} ku^s\bigg)=0, \label{eq:Sommerfeld2}
\end{eqnarray}
where the refractive index $n(x)$ of the inhomogeneous medium is piecewise continuous such that $\mathrm{Re}(n)>0$, $\mathrm{Im}(n)\geq0$ and $1-n(x)$ is supported in $D$.

The direct scattering problems \eqref{eq:Helmholtz}--\eqref{eq:Sommerfeld} and \eqref{eq:Helmholtz_D}--\eqref{eq:Sommerfeld2} admit a unique solution
(see, e.g., \cite{Cakoni, CK13, McLean}), respectively, and the scattered wave $u^s$ has the following asymptotic behavior
$$
u^s(x,d)=\frac{\mathrm{e}^{\mathrm{i} k|x|}}{|x|}\left\{ u^{\infty}(\hat{x},d)+\mathcal{O}\left(\frac{1}{|x|}\right) \right\}, \quad |x|\to\infty
$$
uniformly in all observation directions $\hat{x}=x/|x|\in\mathbb{S}^2$. The analytic function $u^{\infty}(\hat{x},d)$ defined on the unit sphere $\mathbb{S}^2$ is called the far field pattern or scattering amplitude (see \cite{CK13}).

Now, we turn to introducing the inverse acoustic scattering problem for incident point sources with limited-aperture phaseless near-field data. To this end, we first introduce the following definition of admissible surfaces.
\begin{definition}[Admissible surface]\label{def:admissible_surface}
An open surface $\Gamma$ is called an admissible surface with
respect to domain $\Omega$ if

\noindent (i) $\Omega\subset\mathbb{R}^3\backslash\overline{D}$ is bounded and simply-connected;

\noindent (ii) $\partial \Omega$ is analytic homeomorphic to $\mathbb{S}^2$;

\noindent (iii) $k^2$ is not a Dirichlet eigenvalue of $-\Delta$ in $\Omega$;

\noindent (iv) $\Gamma\subset\partial\Omega$ is a two-dimensional analytic manifold with nonvanishing measure.
\end{definition}

\begin{remark}
    We would like to point out that this requirement for the admissibility of $\Gamma$ is quite mild and thus can be easily fulfilled. For instance, $\Omega$ can be chosen as a ball whose radius is less than $\pi/k$ and $\Gamma$ is chosen as an arbitrary corresponding semisphere.
\end{remark}

For a generic point $z\in\mathbb{R}^3\backslash\overline{D}$, the incident field due to the point source located at $z$ is given by
\begin{equation*}
\Phi (x, z):=\frac{\mathrm{e}^{\mathrm{i} k|x-z|}}{4\pi |x-z|}, \quad x\in\mathbb{R}^3\backslash(\overline{D}\cup\{z\}),
\end{equation*}
which is also known as the fundamental solution to the Helmholtz equation. Denote by $v^s_D(x,z)$ and $v^\infty_D(\hat{x},z)$ the near-field and far-field pattern generated by $D$ corresponding to the incident field $\Phi(x, z)$. Define
$$
v(x,z):=v^s_D(x,z)+\Phi(x, z),\quad x\in\mathbb{R}^3\backslash(\overline{D}\cup\{z\})
$$
and
$$
v^{\infty}(\hat{x},z):=v^{\infty}_D(\hat{x},z)+\Phi^\infty(\hat{x}, z),\quad \hat{x}\in\mathbb{S}^2,
$$
where $\Phi^\infty(\hat{x},z):=\mathrm{e}^{-\mathrm{i} k \hat{x}\cdot z}/(4\pi)$ is the the far-field pattern of $\Phi(x, z)$.

For two generic and distinct source points $z_1, z_2\in\mathbb{R}^3\backslash\overline{D}$, we denote by
\begin{equation}\label{incident}
 v^i(x; z_1,z_2):=\Phi(x, z_1)+\Phi(x, z_2),\quad x\in\mathbb{R}^3\backslash(\overline{D}\cup\{z_1\}\cup\{z_2\}),
\end{equation}
the superposition of these point sources. Then, by the linearity of direct scattering problem, the near-field co-produced by $D$ and the incident wave $v^i(x; z_1,z_2)$ is given by
$$
v(x; z_1,z_2):=v(x, z_1)+v(x, z_2), \quad x\in\mathbb{R}^3\backslash(\overline{D}\cup\{z_1\}\cup\{z_2\}).
$$

With these preparations,  we formulate the phaseless inverse scattering problems as the following.

\begin{problem}[Phaseless inverse scattering by a bounded scatterer]\label{prob:obstacle}
Assume that $\Gamma$ and $\Sigma$ are admissible surfaces with respect to $\Omega$ and $G$, respectively, such that
$\overline{\Omega}\cap \overline{G}=\emptyset$. Let $D$ be the impenetrable obstacle with boundary condition $\mathscr{B}$ or the inhomogeneous medium with refractive index $n$. Given the phaseless near-field data
 \begin{equation*}
 \begin{array}{ll}
 & \{|v(x,z_0)|: x\in \Sigma\}, \\
 & \{|v(x,z)|:  x\in \Sigma,\ z\in \Gamma\}, \\
 & \{|v(x,z_0)+v(x,z)|: x\in \Sigma,\ z\in \Gamma\}
 \end{array}
 \end{equation*}
 for a fixed wavenumber $k>0$ and a fixed  $z_0\in\mathbb{R}^3\backslash(\overline{D}\cup\Gamma\cup\Sigma)$, determine the location and shape $\partial D$ as well as the boundary condition
 $\mathscr{B}$ for the obstacle or the refractive index $n$ for the medium inclusion.
\end{problem}

We refer to Figure \ref{fig:illustration} for an illustration of the geometry setting of Problem \ref{prob:obstacle}. The uniqueness of this problem will be analyzed in the next subsection.

\begin{figure}
    \centering
    \newdimen\R % Radius of regular polygons
    \R=0.5cm
    \begin{tikzpicture}[thick]
    \pgfmathsetseed{3}
    \draw plot [smooth cycle, samples=8, domain={1:9},  xshift=0.7cm, yshift=-0.1cm] (\x*360/8+5*rnd:0.5cm+1cm*rnd) [fill=lightgray] node at (0.6,-0.2) {$D$};
    %\pgfmathsetseed{9}
    %\draw plot [smooth cycle, samples=5, domain={1:5}, xshift=2cm] (\x*360/8+5*rnd:0.5cm+1cm*rnd) node at (2,0.5) {$D_2$};
    \draw [blue, dashed] (1.6, 3.4) circle (0.6cm); % plot the reference ball
    \draw node at (1.6,3.5) {$\Omega$};
    \draw (1.05, 3.15) arc(205:345:0.6cm) [very thick, blue];
    \draw node at (2.3,2.9) {$\Gamma$};

    \fill [red] (-1.5,1.5) circle (2pt);  % plot the point source $z_0$
    \draw (-1.5, 1.5) node [above] {$z_0$};
    \draw [->] (-1.4, 1.4)--(-0.8, 0.9);
    \draw (-1.6, 1.2) arc(280:360:0.5cm);
    \draw (-1.6, 1.0) arc(280:360:0.7cm);

    \fill [red] (1.6,2.8) circle (2pt);  % plot the point source $z$
    \draw node at (1.7,3.0) {$z$};
    \draw [->] (1.58, 2.7)--(1.5, 2);
    \draw (1.2, 2.6) arc(225:315:0.5cm);
    \draw (1.1, 2.5) arc(220:320:0.6cm);

    \draw [blue, dashed] (3.6, 1.7) circle (0.7cm); % plot near fields
    \fill [red] (2.96, 1.4) circle (2pt);  % plot the point  $x$
    \draw node at (3.2, 1.5) {$x$};
    \draw (3.02, 2.1) arc(145:265:0.7cm) [very thick, blue];
    \draw node at (3.6, 1.8) {$G$};
    \draw node at (2.8, 2.2) {$\Sigma$};

    \draw (-0.2, 2) node [above] {$v^i(\cdot; z_0, z)$};
    \draw [->] (2.1, 0.8)--(2.8, 1.3);
    \draw (2.5, 0.7) arc(0:70:0.6cm);
    \draw (2.7, 0.7) arc(0:70:0.8cm);

    \draw node at (3.3,0.3) {$|v(\cdot; z_0, z)|$}; % far-field pattern
    \end{tikzpicture}
    \caption{An illustration of the phaseless inverse scattering by a bounded scatterer.} \label{fig:illustration}
\end{figure}
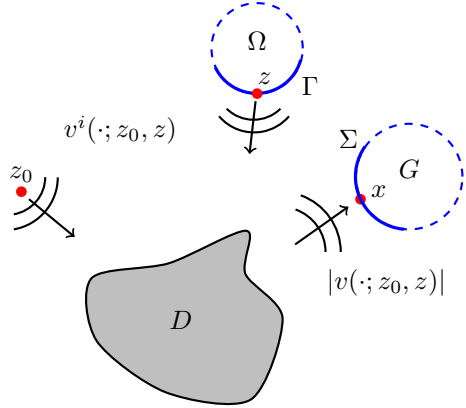

%%%%%%%%%%%%%%%%%%%%%%%%%%%%%%%%%
\subsection{Uniqueness result}\label{subsec:obstacle}
%%%%%%%%%%%%%%%%%%%%%%%%%%%%%%%%%

Now we present the uniqueness results on phaseless inverse scattering. The following theorem shows that Problem \ref{prob:obstacle} admits a unique solution, namely, the geometrical and physical information of the scatterer boundary or the refractive index for the medium can be simultaneously and uniquely determined from the modulus of near-fields.

\begin{theorem}\label{thm:bounded_scatterer}%-------------------------------------------------------------------------
Assume that $\Gamma$ and $\Sigma$ are admissible surfaces with respect to $\Omega$ and $G$, respectively, such that $\overline{\Omega}\cap \overline{G}=\emptyset$.  For two scatterers
$D_1$ and $D_2$, suppose that the corresponding near-fields satisfy that
   \begin{align}
   |v_1(x, z_0)|= & |v_2(x, z_0)|, \quad \forall x  \in \Sigma, \label{obstacle_condition1} \\
   |v_1(x, z)|= & |v_2(x, z)|, \quad \forall (x, z) \in \Sigma\times\Gamma \label{obstacle_condition2}
   \end{align}
   and
   \begin{equation}\label{obstacle_condition3}
   |v_1(x, z_0)+v_1(x, z)|=|v_2(x, z_0)+v_2(x, z)|,\quad \forall (x, z) \in \Sigma\times\Gamma
   \end{equation}
   for an arbitrarily fixed $z_0\in\mathbb{R}^3\backslash(\overline{D}\cup\Gamma\cup\Sigma)$. Then we have

\noindent (i) If $D_1$ and $D_2 $ are two impenetrable obstacles with boundary conditions $\mathscr{B}_1$ and $\mathscr{B}_2$ respectively, then $D_1=D_2$ and $\mathscr{B}_1=\mathscr{B}_2$.

\noindent (ii) If $D_1$ and $D_2 $ are two medium inclusions with refractive indices $n_1$ and $n_2$ respectively, then $n_1=n_2$.
\end{theorem}%-------------------------------------------------------------------------------------------
\begin{proof}
    From  \eqref{obstacle_condition1}, \eqref{obstacle_condition2} and \eqref{obstacle_condition3}, we have for all $x\in\Sigma, z\in\Gamma$
    \begin{equation}\label{Thm1equality1}
    \mathrm{Re}\left\{v_1(x,z_0) \overline{v_1(x,z)}\right\}
    =\mathrm{Re}\left\{v_2(x,z_0) \overline{v_2(x,z)}\right\},
    \end{equation}
    where the overline denotes the complex conjugate. According to \eqref{obstacle_condition1} and \eqref{obstacle_condition2}, we denote
    \begin{equation*}
        v_j(x,z_0)=r(x,z_0) \mathrm{e}^{\mathrm{i} \alpha_j(x,z_0)},\quad
        v_j(x,z)=s(x,z) \mathrm{e}^{\mathrm{i} \beta_j(x,z)},\quad j=1,2,
    \end{equation*}
    where $r(x,z_0)=|v_j(x,z_0)|$, $s(x,z)=|v_j(x,z)|$, $\alpha_j(x,z_0)$ and $\beta_j(x,z)$, are real-valued functions,  $j=1,2$.

    Since $\Sigma$ is an admissible surface of $G$, by Definition \ref{def:admissible_surface} and the analyticity of $v_j(x,z)$ with respect to $x$, we have $s(x,z)\not\equiv 0$ for $x \in \Sigma, z \in\Gamma$. Further, the continuity yields that there exists open sets $\tilde{\Sigma}\subset \Sigma$ and $\Gamma_0\subset\Gamma$
    such that $s(x,z)\neq 0$, $\forall (x,z)\in \tilde{\Sigma}\times\Gamma_0$. Similarly, we have $r(x,z_0)\not\equiv 0$ on $\tilde{\Sigma}$. Again,
    the continuity leads to $r(x,z_0)\neq 0$ on an open set $\Sigma_0 \subset \tilde{\Sigma}$. Therefore, we have $r(x,z_0)\neq 0$, $s(x,z)\neq 0,\ \forall (x, z) \in \Sigma_0\times\Gamma_0$. In addition, taking \eqref{Thm1equality1} into account, we derive that
    \begin{equation*}
        \cos[\alpha_1(x,z_0)-\beta_{1}(x,z)]=\cos[\alpha_2(x,z_0)-\beta_{2}(x,z)], \quad \forall (x, z) \in \Sigma_0\times\Gamma_0.
    \end{equation*}
    Hence, either
    \begin{eqnarray}\label{Thm1equality2}
    \alpha_1(x,z_0)-\alpha_2(x,z_0)=\beta_{1}(x,z)-\beta_{2}(x,z)+ 2m\pi, \quad \forall (x, z) \in \Sigma_0\times\Gamma_0
    \end{eqnarray}
    or
    \begin{eqnarray}\label{Thm1equality3}
    \alpha_1(x,z_0)+\alpha_2(x,z_0)=\beta_1(x,z)+\beta_{2}(x,z)+ 2m\pi, \quad \forall (x, z) \in \Sigma_0\times\Gamma_0
    \end{eqnarray}
    holds with some $m \in \mathbb{Z}$.

    First, we shall consider the case \eqref{Thm1equality2}. Since $z_0$ is fixed, let us define $\gamma(x):=\alpha_1(x,z_0)-\alpha_2(x,z_0)- 2m\pi$
    for $ x \in \Sigma_0$, and then, we deduce for all $ (x, z) \in \Sigma_0\times\Gamma_0$
    \begin{equation*}
        v_1(x,z)=s(x,z)\mathrm{e}^{\mathrm{i} \beta_{1}(x,z)}
        =s(x,z)\mathrm{e}^{\mathrm{i} \beta_{2}(x,z)+\mathrm{i} \gamma(x)}=v_2(x,z)\mathrm{e}^{\mathrm{i} \gamma(x)}.
    \end{equation*}
    From the reciprocity relation \cite[Theorem 3]{AMSS02} for point sources, we have
    \begin{equation*}
        v_1(z,x)= \mathrm{e}^{\mathrm{i} \gamma(x)}v_2(z,x), \quad \forall (x, z) \in \Sigma_0\times\Gamma_0.
    \end{equation*}
Then, for every $x\in \Sigma_0$, by using the analyticity of $v_j(z,x)$($j=1,2$) with respect to $z$, we have $v_1(z,x)= \mathrm{e}^{\mathrm{i} \gamma(x)}v_2(z,x),$   $\forall z\in \partial\Omega$.
Let $w(z,x)=v_1(z,x)-\mathrm{e}^{\mathrm{i} \gamma(x)}  v_2(z,x)$, then
\begin{equation*}
\left\{
\begin{array}{lr}
\Delta w+ k^2 w=0 &\mathrm{in}\ \Omega, \\
w=0  & \mathrm{on}\ \partial \Omega.
\end{array}
\right.
\end{equation*}
By the assumption of $\Omega$ that $k^2$ is not a Dirichlet eigenvalue of $-\Delta$ in $\Omega$,  we find $w=0$ in $\Omega$. Now, the analyticity of
$v_j(z,x)(j=1,2)$ with respect to $z$ yields
    \begin{equation*}
        v_1(z,x)= \mathrm{e}^{\mathrm{i} \gamma(x)}v_2(z,x),\quad \forall z \in \mathbb{R}^3\backslash (\overline{D}_1\cup \overline{D}_2\cup \{x\}).
    \end{equation*}
i.e., for all $z \in \mathbb{R}^3\backslash (\overline{D}_1\cup \overline{D}_2\cup \{x\})$,
\begin{equation}\label{eq:relation}
        v_{D_1}^s(z,x)+\Phi(z,x)=\mathrm{e}^{\mathrm{i} \gamma(x)} \left[ v_{D_2}^s(z,x)+\Phi(z,x)\right].
\end{equation}
By the Green's formula \cite[Theorem 2.5]{CK13}, one can readily deduce that $v^s_{D_j}(z,x)(j=1,2)$ are bounded for $z\in \mathbb{R}^3\backslash (\overline{D}_1\cup \overline{D}_2)$,
which, together with \eqref{eq:relation}, implies that $(1-\mathrm{e}^{\mathrm{i} \gamma(x)})\Phi(z,x)$ is bounded for $z \in \mathbb{R}^3\backslash (\overline{D}_1\cup \overline{D}_2\cup \{x\})$. Hence, by letting $z\to x$, we obtain $\mathrm{e}^{\mathrm{i} \gamma(x)} = 1$, and
\begin{equation*}
    v_{D_1}^s(z,x) =  v_{D_2}^s(z,x),\quad   \forall (x, z) \in \Sigma_0\times \mathbb{R}^3\backslash (\overline{D}_1\cup \overline{D}_2\cup \{x\}).
\end{equation*}
And thus, the far-field patterns coincide, i.e.
\begin{equation*}
       v_{D_1}^\infty(\hat{z},x)= v_{D_2}^\infty(\hat{z},x), \quad \forall (x, \hat{z}) \in \Sigma_0\times\mathbb{S}^2.
\end{equation*}
Now, from the mixed reciprocity relation \cite[Theorem 3.16]{CK13} for the obstacle or \cite[Theorem 2.2.4]{Pot01} for the inhomogeneous medium, we have
\begin{equation*}
       u_1^s(x,-\hat{z})= u_2^s(x,-\hat{z}), \quad \forall (x, \hat{z}) \in \Sigma_0\times \mathbb{S}^2.
\end{equation*}
Further, the analyticity of $u_j^s(x, d)$($j=1,2$) with respect to $x$ yields $u_1^s(x,d)= u_2^s(x,d)$, $\forall (x, d) \in \partial G\times \mathbb{S}^2$.
By the similar discussion of \eqref{eq:relation} for $u_1^s(x,d)- u_2^s(x,d)$ on $G$, we have
\begin{equation*}
    u_1^s(x,d) =  u_2^s(x,d),\quad   \forall (x, d) \in \mathbb{R}^3\backslash (\overline{D}_1\cup \overline{D}_2)\times \mathbb{S}^2.
\end{equation*}
Therefore, we obtain
    \begin{eqnarray}\label{coincide}
    u^{\infty}_1(\hat{x},d)= u^{\infty}_2(\hat{x},d),
    \quad \forall \hat{x}, d \in \mathbb{S}^2.
    \end{eqnarray}

Next we are going to show that the case \eqref{Thm1equality3} does not hold. Suppose that \eqref{Thm1equality3} is true, then following a similar argument, we see that there exists $\eta(x)$ such that $v_1(z,x)=\mathrm{e}^{\mathrm{i} \eta(x)} \overline{v_2(z,x)}$ for $x\in \Sigma_0$, $z \in \mathbb{R}^3\backslash (\overline{D}_1\cup \overline{D}_2\cup \{x\})$, i.e.
\begin{equation*}
        v_{D_1}^s(z,x)+\Phi(z,x)=\mathrm{e}^{\mathrm{i} \eta(x)}  \overline{\left[v_{D_2}^s(z,x) +\Phi(z,x)\right]}.
\end{equation*}
Then, from the boundedness of $v_{D_j}^s(z,x)$, it can be seen that $\Phi(z,x)-\mathrm{e}^{\mathrm{i} \eta(x)} \overline{\Phi(z,x)}$ is bounded for all $z \in \mathbb{R}^3\backslash (\overline{D}_1\cup \overline{D}_2\cup \{x\})$. Since
\begin{equation*}
        \Phi(z,x)-\mathrm{e}^{\mathrm{i} \eta(x)} \overline{\Phi(z,x)}
      =  [1-\mathrm{e}^{\mathrm{i} \eta(x)}]\frac{\cos (k|z-x|)}{4\pi|z-x|}\\
      +\mathrm{i}[1+\mathrm{e}^{\mathrm{i} \eta(x)}]\frac{\sin (k|z-x|) }{4\pi|z-x|},
\end{equation*}
by letting $z\to x$, we deduce $\mathrm{e}^{\mathrm{i} \eta(x)}=1$, and thus, $v_1(z,x)= \overline{v_2(z, x)}$ for $z\in\mathbb{R}^3\backslash (\overline{D}_1\cup \overline{D}_2\cup \{x\})$. We claim that $v_1^\infty(\hat{z},x)\not\equiv0$ for
 $\hat{z}\in \mathbb{S}^2$. Otherwise, if $v_1^\infty(\hat{z},x) \equiv0$ for $\hat{z}\in \mathbb{S}^2$, then from Rellich Lemma \cite[Theorem 2.14]{CK13}, we have $v_1(z, x)=0$ for all $z \in \mathbb{R}^3\backslash (\overline{D}_1\cup \{x\})$. Further, from
 $v_1(z, x)=v_{D_1}^s(z, x)+\Phi(z, x)$ and the boundedness of $v_{D_1}^s(z, x)$, we deduce $\Phi(z, x)$ is bounded for all
 $z \in \mathbb{R}^3\backslash (\overline{D}_1\cup \{x\})$, which is a contradiction. Then, the continuity leads to
 $v_1^\infty(\hat{z},x)\neq0$  $\forall\hat{z}\in S$, where $S\subset \mathbb{S}^2$ is an open set.  By taking $\tilde{z}\in S$, $z=\rho\tilde{z}$,
 and using the definition of far-field pattern (see \cite[Theorem 2.6]{CK13}), we obtain
\begin{equation*}
        \lim\limits_{\rho\to\infty} \rho\mathrm{e}^{-\mathrm{i} k\rho} v_1(\rho\tilde{z},x)=v_1^\infty(\tilde{z},x)
\end{equation*}
and
\begin{equation*}
        \lim\limits_{\rho\to\infty} \rho\mathrm{e}^{\mathrm{i} k\rho} \overline{v_2(\rho\tilde{z},x)}=\overline{v_2^\infty(\tilde{z},x)}.
\end{equation*}
 Further, noticing $v_1(\rho\tilde{z},x)= \overline{v_2(\rho\tilde{z},x)}$ and $v_1^\infty(\tilde{z},x)\neq0$, we have
\begin{equation*}
        \lim\limits_{\rho\rightarrow \infty}  \mathrm{e}^{2\mathrm{i} k\rho}  = \frac{\overline{v_2^\infty(\tilde{z},x)}}{v_1^\infty(\tilde{z},x)},
\end{equation*}
which is a contradiction. Hence, the case \eqref{Thm1equality3} does not hold.

Having verified \eqref{coincide}, we shall complete our proof as the consequences of two existing uniqueness results. For the inverse obstacle scattering, by Theorem 5.6 in \cite{CK13}, we have $D_1=D_2$ and $\mathscr{B}_1=\mathscr{B}_2$, and for inverse medium scattering, Theorem 10.5 in \cite{CK13} leads to  $n_1=n_2$.
\end{proof}

\begin{remark}
    We would like to point out that an analogous uniqueness result in two dimensions remains valid after appropriate modifications of the fundamental solution, the radiation condition and the admissible surface. So we omit the 2D details.
\end{remark}

\begin{remark}
   We would like to remark that a similar result on uniqueness can be also obtained by using the superposition of a fixed plane wave and some point sources as the incident fields.
\end{remark}

%%%%%%%%%%%%%%%%%%%%%%%%%%%%%%%%%%%%%%%%%%%%%%%%%%%%%%%%%
\section{Uniqueness for inverse scattering by locally perturbed half-planes}\label{sec:half-plane}
%%%%%%%%%%%%%%%%%%%%%%%%%%%%%%%%%%%%%%%%%%%%%%%%%%%%%%%%%

\subsection{Problem statement}

We begin this section with the precise formulations of the model scattering problem. Assume that the real-valued function $f\in C^2(\mathbb{R})$ has a compact support. Let $\Gamma=\{x=(x_1, x_2)\in \mathbb{R}^2 |\  x_2=f(x_1), x_1\in \mathbb{R} \}$ be the locally perturbed curve and $D=\{x\in \mathbb{R}^2|\ x_2>f(x_1), x_1\in \mathbb{R} \}$ be the locally perturbed half-plane above curve $\Gamma$. Denote by $\Gamma_c=\{x\in \mathbb{R}^2 |\ x_2=0\}$ and by $\Gamma_p=\Gamma\backslash\Gamma_c$ the local perturbation.  For a generic point $z\in D$, the incident field $u^i$ due to the point source located at $z$ is given by
\begin{equation*}
u^i (x, z):=\frac{\mathrm{i} }{4}H_0^{(1)}(k|x-z|)=\frac{\mathrm{i} }{4}J_0(k|x-z|)-\frac{1 }{4}Y_0(k|x-z|), \, x\in D\backslash\{z\},
\end{equation*}
which is also known as the fundamental solution to the Helmholtz equation with wavenumber $k>0$, where $J_0$ and $Y_0$ are the Bessel functions of the first kind and the second kind of order 0, respectively.  Then, the scattering problem can be formulated as: find the scattered field $u^s$, such that
%\begin{eqnarray}
%\Delta u^s+ k^2 u^s= 0\quad \mathrm{in}\ D,\label{eq:Helmholtz2} \\
%\mathscr{B}_c u= 0 \quad \mathrm{on}\ \Gamma_c, \label{eq:boundary_condition1}\\
%\mathscr{B}_p u= 0 \quad \mathrm{on}\ \Gamma_p, \label{eq:boundary_condition2}\\
%\lim\limits_{r=|x|\to\infty} \sqrt{r}\bigg(\displaystyle\frac{\partial u^s}{\partial r} -\mathrm{i} ku^s\bigg)=0, \label{eq:Sommerfeld3}
%\end{eqnarray}
\begin{align}
\Delta u^s+ k^2 u^s= & 0\quad \mathrm{in}\ D,\label{eq:Helmholtz2} \\
\mathscr{B}_c u= & 0 \quad \mathrm{on}\ \Gamma_c, \label{eq:boundary_condition1}\\
\mathscr{B}_p u= & 0 \quad \mathrm{on}\ \Gamma_p, \label{eq:boundary_condition2}\\
\lim\limits_{r=|x|\to\infty} \sqrt{r}\bigg(\dfrac{\partial u^s}{\partial r} -&\mathrm{i} ku^s\bigg)=0, \label{eq:Sommerfeld3}
\end{align}
where $u=u^i+u^s$ denotes the total field. Here $\mathscr{B}_c$ and $\mathscr{B}_p$ in \eqref{eq:boundary_condition1}-\eqref{eq:boundary_condition2} are the boundary operators defined by
\begin{eqnarray}\label{BCC}
\mathscr{B}_cu=\left\{
\begin{array}{ll}
u, & \text{for a sound-soft perturbed half-plane},  \\
\displaystyle\frac{\partial u}{\partial \nu}, & \text{for a sound-hard perturbed half-plane},
\end{array}
\right.
\end{eqnarray}
and
\begin{eqnarray}\label{BCP}
\mathscr{B}_pu=\left\{
\begin{array}{ll}
u, & \mathrm{on}\ \Gamma_{p,_D},  \\
\displaystyle\frac{\partial u}{\partial \nu}+\lambda u, & \mathrm{on}\ \Gamma_{p,_I},
\end{array}
\right.
\end{eqnarray}
where $\nu$ is the unit normal on $\Gamma$ directed into $D$, $\Gamma_{p,_D}\cup \Gamma_{p,_I}=\Gamma_{p}$, $\Gamma_{p,_D}\cap \Gamma_{p,_I}=\emptyset$,
$\lambda\in C(\Gamma_{p,_I})$ and $\mathrm{Im}\lambda\geq0$. The mixed boundary condition \eqref{BCP} is rather general in the sense that it covers the usual Dirichlet/sound-soft boundary condition ($\Gamma_{p,_I}=\emptyset$), the Neumann/sound-hard boundary condition ($\Gamma_{p,_D}=\emptyset$ and $\lambda=0$), and the impedance boundary condition ($\Gamma_{p,_D}=\emptyset$ and $\lambda\neq 0$).

The existence of a unique solution to the scattering problem \eqref{eq:Helmholtz2}--\eqref{eq:Sommerfeld3} by a sound-soft perturbed half-plane with $\Gamma_{p,_I}=\emptyset$ was established in \cite{ZhangDeyue2008, BL2011} by a variational method or in \cite{ZHW2013} by an integral equation method, and the well-posedness of the problem \eqref{eq:Helmholtz2}--\eqref{eq:Sommerfeld3} by a sound-hard perturbed half-plane with $\Gamma_{p,_D}=\emptyset$ and $\lambda=0$ was studied in \cite{Qu} by the integral equation method. In three-dimension case, we refer to \cite{Willers} for the integral equation method.

The well-posedness of the problem \eqref{eq:Helmholtz2}--\eqref{eq:Sommerfeld3} by a sound-soft perturbed half-plane with $\Gamma_{p,_I}\neq\emptyset$ can be obtained similarly by the variational method as shown in \cite{ZhangDeyue2008, BL2011}, while the existence and uniqueness of solutions to the problem \eqref{eq:Helmholtz2}--\eqref{eq:Sommerfeld3} by a sound-hard perturbed half-plane can be established by the variational method with an even expansion and extension of the solution (an odd expansion and extension of the solution in \cite{ZhangDeyue2008, BL2011}) and a simple modification of the Dirichlet-to-Neumann operator in \cite{ZhangDeyue2008, BL2011} which does not affect the properties.

In the following we are going to consider the inverse scattering problem by the locally perturbed half-plane for incident point sources with limited-aperture phaseless near-field data. Similar to Definition \ref{def:admissible_surface}, we first introduce the description of admissible curves.
\begin{definition}[Admissible curve]\label{def:admissible_curve}
	An open curve $\Lambda$ is called an admissible curve with respect to domain $\Omega$ if
	
	\noindent (i) $\overline{\Omega}\subset D $ is bounded and simply-connected;
	
	\noindent (ii) $\partial \Omega$ is analytic homeomorphic to $\mathbb{S}$;
	
	\noindent (iii) $k^2$ is not a Dirichlet eigenvalue of $-\Delta$ in $\Omega$;
	
	\noindent (iv) $\Lambda\subset\partial\Omega$ is a one-dimensional analytic manifold with nonvanishing measure.
\end{definition}

\begin{remark}
	It can be readily seen that Definition \ref{def:admissible_curve} is the one-dimensional version of Definition \ref{def:admissible_surface}. Thus it is easy to find an admissible pair of $(\Omega, \Lambda)$. For example, $\Omega$ can be chosen as a disk whose radius is less than $2.4048/k$ and $\Lambda$ is chosen as an arbitrary corresponding semicircle.
\end{remark}

Analogous to the arguments in the previous section, for two generic and distinct source points $z_1, z_2\in D$, we denote by
\begin{equation}\label{incident2}
u^i(x; z_1,z_2):=u^i(x, z_1)+u^i(x, z_2),\quad x\in D\backslash(\{z_1\}\cup\{z_2\}),
\end{equation}
the superposition of these point sources. Then, by the linearity of direct scattering problem, the total near-field is given by
$$
u(x; z_1,z_2):=u(x, z_1)+u(x, z_2), \quad x\in D\backslash(\{z_1\}\cup\{z_2\}).
$$

We are now in the position to formulate the phaseless inverse scattering problems under consideration.

\begin{problem}[Phaseless inverse scattering by locally perturbed half-planes]\label{prob:half-plane}
	Let $\Gamma$ be the locally perturbed curve with boundary condition $\mathscr{B}_c$ and $\mathscr{B}_p$. Assume that $\Lambda$ and $\Sigma$ are admissible curves with respect to $\Omega$ and $G$, respectively.  Given the phaseless
	near-field data
	\begin{equation*}
	\begin{array}{ll}
	& \{|u(x,z_0)|: x\in \Sigma\}, \\
	& \{|u(x,z)|:  x\in \Sigma,\ z\in \Lambda\}, \\
	& \{|u(x,z_0)+u(x,z)|: x\in \Sigma,\ z\in \Lambda\},
	\end{array}
	\end{equation*}
	for a fixed wavenumber $k>0$ and a fixed  $z_0\in D\backslash( \Lambda\cup\Sigma)$, determine the locally perturbed curve $\Gamma$ as well as the boundary condition $\mathscr{B}_c$ and $\mathscr{B}_p$.
\end{problem}

For an illustration of the above problem, we refer to  Figure \ref{fig:half-plane}. The next subsection will be devoted to the uniqueness issue of this problem.

\begin{figure}
	\centering
	\newdimen\R % Radius of regular polygons
	\R=0.5cm
	\begin{tikzpicture}[thick]
	
	\draw[-](2.8, -1.3)--(6.3, -1.3)[very thick, blue];
	\draw[-](-3.3, -1.3)--(0.02, -1.3) [very thick, blue];
	\draw node at (4.6,-1.05) {$\Gamma_c$};
	\draw node at (-2.,-1.05) {$\Gamma_c$};
	\draw node at (1.15,-1.05) {$\Gamma_p$};
	\draw node at (1.3, 0.2) {$D$};
	
	\draw [blue, dashed] (-1.3, 1) circle (0.7cm); % plot the reference ball
	\draw node at  (-1.3, 1.1){$\Omega$};
	\draw (-1.3, 0.3) arc(270:370:0.7cm) [very thick, blue];
	\draw node at (-1.2,0.05) {$\Lambda$};
	\draw (-0.95, 0.43) node [above] {$z$};
	\fill [red] (-0.8,0.5) circle (2pt);  % plot the point source $z_0$
	\draw [->] (-0.72, 0.4)--(-0.2, -0.1);
	\draw (-0.8, 0.1) arc(280:360:0.5cm);
	\draw (-0.8, -0.1) arc(280:360:0.7cm);
	
	\fill [red] (1.3,1.8) circle (2pt);  % plot the point source $z$
	\draw node at (1.37,2.1) {$z_0$};
	\draw [->] (1.29, 1.67)--(1.2, 1);
	\draw (0.88, 1.6) arc(225:315:0.5cm);
	\draw (0.8, 1.5) arc(220:320:0.6cm);
	
	\draw [blue, dashed] (3.9, 1.2) circle (0.7cm); % plot near fields
	\fill [red] (3.23, 1.) circle (2pt);  % plot the point  $x$
	\draw node at (3.5, 1.) {$x$};
	\draw (3.32, 1.6) arc(145:265:0.7cm) [very thick, blue];
	\draw node at (3.9, 1.3) {$G$};
	\draw node at (3.13, 1.77) {$\Sigma$};
	
	\draw (-0.1, 1.5) node [above] {$u^i(\cdot; z_0, z)$};
	\draw [->] (2.55, 0.3)--(3.15, 0.9);
	\draw (2.9, 0.3) arc(20:70:0.6cm);
	\draw (3.1, 0.3) arc(10:80:0.7cm);
	
	\draw node at (4.15, 0.15) {$|u(\cdot; z_0, z)|$}; % near-field data
	
	\draw [red, very thick] (0.02, -1.3) to [out=0, in=180] (1, -0.5) to [out=0, in=180] (2.2, -2) to [out=0, in=180] (2.8, -1.3); % curve $\Gamma_p$
	\end{tikzpicture}
	\caption{An illustration of the phaseless inverse scattering by a locally perturbed half-plane.} \label{fig:half-plane}
\end{figure}
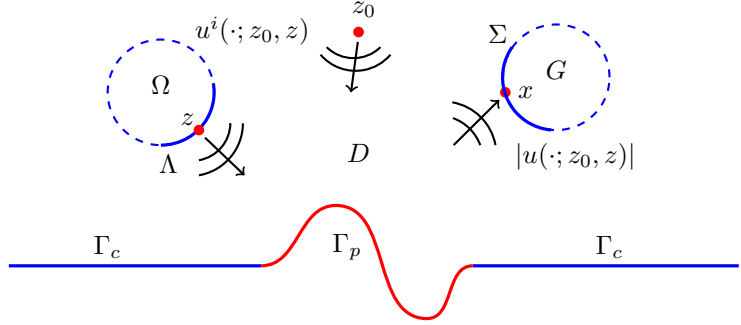

%%%%%%%%%%%%%%%%%%%%%%%%%%%%%%%%%%%%%%%%%%%%%%
\subsection{Uniqueness results}
%%%%%%%%%%%%%%%%%%%%%%%%%%%%%%%%%%%%%%%%%%%%%%

Before we present the uniqueness result on phaseless inverse scattering, the following reciprocity relation for the total fields is needed.

\begin{lemma}\label{Lemma1}%-------------------------------------------------------------------------
	Let $u^s(x, z)$ be the scattered field satisfying \eqref{eq:Helmholtz2}--\eqref{eq:Sommerfeld3}. Then we have
	\begin{eqnarray}\label{reciprocity}
	u(x, z)=u(z, x), \quad \forall x, z\in D,\ x\neq z.
	\end{eqnarray}
\end{lemma}%-------------------------------------------------------------------------------------------
\begin{proof}
	The proof of the reciprocity relation is similar to that of Theorem 3.1.4 in \cite{Lines}, so the details are omitted. 
\end{proof}

Let $\Gamma_j=\{x\in \mathbb{R}^2|\ x_2=f_j(x_1), x_1\in \mathbb{R} \}$ be the locally perturbed curve with the real-valued function $f_j\in C^2(\mathbb{R})$ having a compact support, $j=1,2$. Denote by $D_j=\{x\in \mathbb{R}^2|\ x_2>f_j(x_1), x_1\in \mathbb{R} \}$ the domain above $\Gamma_j$, $j=1,2$, and by $D_0=D_1\cap D_2$.  

Denote by $u_j^s$ and $u_j$ the scattered field and the total field generated by $\Gamma_j$, respectively, corresponding to the incident field $u^i(x, z)$, $j=1,2$. Now, the following theorem shows that Problem \ref{prob:half-plane} admits a unique solution, namely, the geometrical and physical information of the locally perturbed plane can be simultaneously and uniquely determined from the modulus of total near-fields.

\begin{theorem}\label{thm:half-plane}%-------------------------------------------------------------------------
	Let $\Gamma_1$ and $\Gamma_2 $ be two locally perturbed curves with boundary conditions $\mathscr{B}_{c,1}, \mathscr{B}_{p,1}$ and $\mathscr{B}_{c,2}, \mathscr{B}_{p,2}$, respectively. Assume that $\Lambda$ and $\Sigma$ are admissible curves with respect to $\Omega$ and $G$, respectively, such that $\overline{\Omega}\subset\subset D_0$, $\overline {G}\subset\subset D_0$ and $\overline{\Omega}\cap \overline {G}=\emptyset$. Suppose that the corresponding total near-fields satisfy that
	\begin{align}
	|u_1(x,z_0)|= & |u_2(x,z_0)|, \quad \forall x  \in \Sigma, \label{plane_condition1} \\
	|u_1(x,z)|= & |u_2(x,z)|, \quad \forall (x, z) \in \Sigma\times\Lambda \label{plane_condition2}
	\end{align}
	and
	\begin{equation}\label{plane_condition3}
	|u_1(x,z_0)+u_1(x,z)|=|u_2(x,z_0)+u_2(x,z)|,\quad \forall (x, z) \in \Sigma\times\Lambda
	\end{equation}
	for an arbitrarily fixed $z_0\in D_0  \backslash(\Lambda\cup\Sigma)$. Then we have  $\Gamma_1=\Gamma_2$, $\mathscr{B}_{c,1}=\mathscr{B}_{c,2}$ and $\mathscr{B}_{p,1}=\mathscr{B}_{p,2}$.
\end{theorem}%-------------------------------------------------------------------------------------------

\begin{proof}
	In terms of \eqref{plane_condition1}, \eqref{plane_condition2} and \eqref{plane_condition3}, we have for all $x\in\Sigma, z\in\Lambda$
	\begin{equation}\label{Thm2equality1}
	\mathrm{Re}\left\{u_1(x,z_0) \overline{u_1(x,z)}\right\}
	=\mathrm{Re}\left\{u_2(x,z_0) \overline{u_2(x,z)}\right\}.
	\end{equation}
	Using \eqref{plane_condition1} and \eqref{plane_condition2}, we denote
	\begin{equation*}
	u_j(x,z_0)=r(x,z_0) \mathrm{e}^{\mathrm{i} \alpha_j(x,z_0)},\quad
	u_j(x,z)=s(x,z) \mathrm{e}^{\mathrm{i} \beta_j(x,z)},\quad j=1,2,
	\end{equation*}
	where $r(x,z_0)=|u_j(x,z_0)|$, $s(x,z)=|u_j(x,z)|$, $\alpha_j(x, z_0)$ and $\beta_j(x, z)$, are real-valued functions,  $j=1,2$.
	
	Due to the fact that $\Sigma$ is an admissible curve of $G$, Definition \ref{def:admissible_curve} and the analyticity of $u_j(x,z)$ with respect to $x$ imply that $s(x,z)\not\equiv 0$
	for $x \in \Sigma, z \in\Lambda$. Moreover, by the continuity we deduce that there exists open sets $\tilde{\Sigma}\subset \Sigma$ and $\Lambda_0\subset\Lambda$
	such that $s(x, z)\neq 0$, $\forall (x,z)\in \tilde{\Sigma}\times\Lambda_0$. Analogously, we obtain $r(x,z_0)\not\equiv 0$ on $\tilde{\Sigma}$. The continuity also leads to $r(x, z_0)\neq 0$ on an open set $\Sigma_0 \subset \tilde{\Sigma}$. Hence, we have $r(x, z_0)\neq 0$, $s(x,z)\neq 0,\ \forall (x, z) \in \Sigma_0\times\Lambda_0$. Furthermore, we derive from \eqref{Thm2equality1} that
	\begin{equation*}
	\cos[\alpha_1(x,z_0)-\beta_{1}(x,z)]=\cos[\alpha_2(x,z_0)-\beta_{2}(x,z)], \quad \forall (x, z) \in \Sigma_0\times\Lambda_0.
	\end{equation*}
	Therefore, either
	\begin{eqnarray}\label{Thm2equality2}
	\alpha_1(x,z_0)-\alpha_2(x,z_0)=\beta_{1}(x,z)-\beta_{2}(x,z)+ 2m\pi, \;\forall (x, z) \in \Sigma_0\times\Lambda_0
	\end{eqnarray}
	or
	\begin{eqnarray}\label{Thm2equality3}
	\alpha_1(x,z_0)+\alpha_2(x,z_0)=\beta_1(x,z)+\beta_{2}(x,z)+ 2m\pi, \; \forall (x, z) \in \Sigma_0\times\Lambda_0
	\end{eqnarray}
	holds with some $m \in \mathbb{Z}$.
	
	We first deal with the case \eqref{Thm2equality2}. Note that $z_0$ is fixed, we can define $\gamma(x):=\alpha_1(x, z_0)-\alpha_2(x, z_0)- 2m\pi$
	for $ x \in \Sigma_0$, thus we deduce for all $ (x, z) \in \Sigma_0\times\Lambda_0$
	\begin{equation*}
	u_1(x, z)=s(x, z)\mathrm{e}^{\mathrm{i} \beta_1(x,z)}
	=s(x,z)\mathrm{e}^{\mathrm{i} \beta_{2}(x,z)+\mathrm{i} \gamma(x)}=u_2(x,z)\mathrm{e}^{\mathrm{i} \gamma(x)}.
	\end{equation*}
	By the reciprocity relation \eqref{reciprocity}, we arrive at
	\begin{equation*}
	u_1(z,x)= \mathrm{e}^{\mathrm{i} \gamma(x)}u_2(z,x), \quad \forall (x, z) \in \Sigma_0\times\Lambda_0.
	\end{equation*}
	Now, for every $x\in \Sigma_0$, the analyticity of $u_j(z, x)$($j=1,2$) with respect to $z$ leads to $u_1(z, x)= \mathrm{e}^{\mathrm{i} \gamma(x)}u_2(z, x),$   $\forall z\in \partial\Omega$.
	Define $w(z, x)=u_1(z, x)-\mathrm{e}^{\mathrm{i} \gamma(x)}  u_2(z,x)$, then
	\begin{equation}\label{eigenvalue}
	\left\{
	\begin{array}{lr}
	\Delta w+ k^2 w=0 &\mathrm{in}\ \Omega, \\
	w=0  &\mathrm{on}\ \partial \Omega.
	\end{array}
	\right.
	\end{equation}
	Since $k^2$ is not a Dirichlet eigenvalue of $-\Delta$ in $\Omega$, \eqref{eigenvalue} implies that $w=0$ in $\Omega$. So, the analyticity of $u_j(z,x)(j=1,2)$ with respect to $z$ yields
	\begin{equation*}
	u_1(z, x)= \mathrm{e}^{\mathrm{i} \gamma(x)}u_2(z,x),\quad \forall z \in D_0\backslash \{x\}.
	\end{equation*}
	namely, for all $ z \in D_0\backslash \{x\}$,
	\begin{equation}\label{eq:relation2}
	u^s_1(z, x)+u^i(z, x)=\mathrm{e}^{\mathrm{i} \gamma(x)}[u^s_2(z, x)+u^i(z, x)].
	\end{equation}
	From the boundedness of $u^s_j(z,x)$ for $z\in D_j$ $(j=1,2)$  and \eqref{eq:relation}, we see $(1-\mathrm{e}^{\mathrm{i} \gamma(x)})u^i(z, x)$ is bounded for $z \in D_0\backslash \{x\}$. Therefore, by letting $z\rightarrow x$, we find $\mathrm{e}^{\mathrm{i} \gamma(x)} = 1$. Furthermore, the continuity of the scattered fields and the reciprocity relation \eqref{reciprocity} lead to
	\begin{equation*}
	u_{1}^s(x, z) =  u_{2}^s(x, z), \quad  x\in \Sigma_0, z \in D_0.
	\end{equation*}
	By a similar argument of \eqref{eigenvalue} for $\varpi(x, z)=u_1^s(x, z)- u_2^s(x, z)$ on $G$, we have
	\begin{equation}\label{coincide2}
	u_{1}^s(x, z) =  u_{2}^s(x, z),\quad   \forall x, z \in D_0.
	\end{equation}
	
	Next we shall show that the case \eqref{Thm2equality3} does not hold. Suppose that \eqref{Thm2equality3} is true, then following a similar argument, we find that for every $x\in \Sigma_0$, there exists $\eta(x)$ such that $u_1(z,x)=\mathrm{e}^{\mathrm{i} \eta(x)}\overline{u_2(z,x)}$ for all $z \in D_0\backslash  \{x\}$, that means
	\begin{equation*}
	u_1^s(z,x)+u^i(z, x)=\mathrm{e}^{\mathrm{i} \eta(x)}  \overline{[u_{2}^s(z,x) +u^i(z,x)]}.
	\end{equation*}
	Then, it can be seen from the boundedness of $u_{j}^s(z,x)$ that $u^i(z,x)-\mathrm{e}^{\mathrm{i} \eta(x)} \overline{u^i(z,x)}$ is bounded for all $z \in D_0\backslash  \{x\}$. By letting $z\to x$ in
	\begin{equation*}
		u^i(z, x)-\mathrm{e}^{\mathrm{i} \eta(x)} \overline{u^i(z,x)}
		= \frac{\mathrm{e}^{\mathrm{i} \eta(x)}-1}{4}Y_0(k|x-z|) +\mathrm{i}\frac{\mathrm{e}^{\mathrm{i} \eta(x)}+1}{4}J_0(k|x-z|),
	\end{equation*}
	we find that $\mathrm{e}^{\mathrm{i} \eta(x)}=1$, which implies $u_1(z, x)= \overline{u_2(z, x)}$ for $z \in D_0\backslash \{x\}$. Now, following the same ideas in Theorem \ref{thm:bounded_scatterer}, we deduce that the case \eqref{Thm2equality3} does not hold.
	
	Once the case \eqref{coincide2} is verified, we would conclude that $\Gamma_1=\Gamma_2$. Otherwise, assume that $\Gamma_1\neq\Gamma_2$. Then, without loss of generality, there exists $x^* \in \partial D_0$ such that $x^* \in \Gamma_1$ and $x^* \in D_2$. Define
	\begin{equation*}
	z_n:=x^*-\frac{1}{n} \nu(x^*),  \quad n=1,2,...
	\end{equation*}
	such that $z_n\in D_0$ for sufficiently large $n$.  Then, from the reciprocity relation and the smoothness of $u_2^s(x^*,z)$ in $D_2$, we have
	\begin{equation*}
	\lim_{n\to\infty}  B_{p,1} u_2^s(x^*,z_n) =  \lim_{n\to\infty}  B_{p,1} u_2^s(z_n, x^*) = B_{p,1}u_2^s(x^*,x^*).
	\end{equation*}
	On the other hand, the boundary conditions \eqref{eq:boundary_condition2} and \eqref{coincide2} imply that
	\begin{equation*}
	\lim_{n\to\infty}  B_{p,1} u_2^s(x^*,z_n) =\lim_{n\rightarrow \infty}B_{p,1} u_1^s(x^*,z_n)=-\lim_{n\to\infty}B_{p,1} u^i(x^*,z_n)=\infty,
	\end{equation*}
	which is a contradiction. Therefore $\Gamma_1=\Gamma_2$.  Further, similar to the proof of Theorem 5.6 in \cite{CK13}, we obtain $\mathscr{B}_{c,1}=\mathscr{B}_{c,2}$ and $\mathscr{B}_{p,1}=\mathscr{B}_{p,2}$.
\end{proof}

\begin{remark}
We want to point out that an analogous uniqueness result in three dimensions remains valid subject to some modifications of the fundamental solution and the admissible curve. 
\end{remark}

%%%%%%%%%%%%%%%%%%%%%%%%%%%%%
%\section{Conclusion}\label{sec:conclusion}
%%%%%%%%%%%%%%%%%%%%%%%%%%%%%

%%%%%%%%%%%%%%%%%%%%%%%%%%%%%%%%
\section*{Acknowledgements}
%%%%%%%%%%%%%%%%%%%%%%%%%%%%%%%%

D. Zhang and F. Sun were supported by NSFC grant 11671170 and Y. Guo was supported by NSFC grants 11601107, 41474102 and 11671111. The work of H. Liu was supported by the FRG and startup grants from Hong Kong Baptist University, Hong
123 Kong RGC General Research Funds, 12302415 and 12302017.

%The authors would also like to thank the anonymous referees for their valuable comments and suggestions, which have led to great improvements on our manuscript.

%The authors would like to thank the referees for their careful
%reading and valuable comments which improved the quality of our
%submitted manuscript

%%%%%%%%%%%%%%%%%%%%%%%%%%%%%%%%
%\section*{References}


\begin{thebibliography}{99}

\bibitem{ACZ16} Ammari H, Chow Y T and Zou J 2016 Phased and phaseless domain reconstructions in the inverse scattering problem via scattering coefficients {\it SIAM J. Appl. Math.} {\bf 76} 1000--1030

\bibitem{AMSS02} Athanasiadis C, Martin P, Spyropoulos A and Stratis I 2002 Scattering relations for point sources: Acoustic and electromagnetic waves  {\it J. Math. Phys} {\bf 43} 5683--5697

\bibitem{BL2011} Bao G and Lin J 2011 Imaging of local surface displacement on an infinite ground plane: the multiple frequency case {\it SIAM J. Appl. Math.} {\bf 71} 1733--1752

\bibitem{BLL13} Bao G, Li P and Lv J 2013 Numerical solution of an inverse diffraction grating problem from phaseless data {\it J. Opt. Soc. Am. A} {\bf 30} 293--299

\bibitem{BZ16} Bao G and Zhang L 2016 Shape reconstruction of the multi-scale rough surface from multi-frequency phaseless data {\it Inverse Problems} {\bf 32} 085002

\bibitem{Cakoni} Cakoni F, Colton D  and Monk P 2001 The direct and inverse scattering problem for partially coated obstacles  {\it Inverse Problems} {\bf 17} 1997--2015

\bibitem{Candes} Candes E J, Strohmer T and Voroninski V 2013 PhaseLift: Exact and stable signal recovery from magnitude measurements via convex programming {\it Commun. Pure Appl. Math.} {\bf 66} 1241--1274.

\bibitem{Candes1} Candes E J, Li X and Soltanolkotabi M 2015 Phase retrieval via Wirtinger flow: Theory and algorithms {\it IEEE Trans. Information Theory} {\bf 61} 1985--2007

\bibitem{Caorsi} Caorsi S, Massa A, Pastorino M and Randazzo A 2003 Electromagnetic detection of dielectric scatterers using phaseless synthetic and real data and the memetic algorithm {\it IEEE Trans. Geosci. Remote Sensing} {\bf 41}  2745--2753

\bibitem{CH17} Chen Z and Huang G 2017 Phaseless imaging by reverse time migration: acoustic waves {\it Numer. Math. Theor. Meth. Appl.} {\bf 10} 1--21

\bibitem{CFH17} Chen Z, Fang S and Huang G 2017 A direct imaging method for the half-space inverse scattering problem with phaseless data {\it Inverse Probl. Imaging} {\bf 11} 901--916

%\bibitem{Colton1} Colton D, Coyle J and Monk P 2000 Recent developments in inverse acoustic scattering theory {\it SIAM Rev.} {\bf 42} 369--414

%\bibitem{Colton2} Colton D, Giebermann K and Monk P 2000 A regularized sampling method for solving three-dimensional inverse scattering problems {\it SIAM J. Sci. Comput.} {\bf 21} 2316--2330

\bibitem{CK13} Colton D and Kress R 2013 {\it Inverse Acoustic and Electromagnetic Scattering Theory} {\it 3rd ed}. (New York: Springer-Verlag)

\bibitem{DZG19} Dong H, Zhang D and Guo Y 2019 A reference ball based iterative algorithm for imaging acoustic obstacle from phaseless far-field data {\it Inverse Problems and Imaging} {\bf 13} 177--195

\bibitem{DLL18} Dong H, Lai J and Li P 2019 Inverse obstacle scattering problem for elastic waves with phased or phaseless far-field data. {\it SIAM J. Imaging Sci.} 12(2), 809--838

\bibitem{GDM18} Gao P, Dong H and Ma F 2018 Inverse scattering via nonlinear integral equations method for a sound-soft crack from phaseless data {\it Applications of Mathematics}, {\bf 63} 149--165

\bibitem{Iva07} Ivanyshyn O 2007 Shape reconstruction of acoustic obstacles from the modulus of the far field pattern  {\it Inverse Probl. Imaging}  {\bf 1} 609--622

\bibitem{Iva10} Ivanyshyn O and Kress R 2010 Identification of sound-soft 3D obstacles from phaseless data {\it Inverse Probl. Imaging}  {\bf 4}  131--149

\bibitem{Iva11} Ivanyshyn O and Kress R 2011 Inverse scattering for surface impedance from phaseless far field data {\it J. Comput. Phys.}  {\bf 230}  3443--3452

\bibitem{JL18} Ji X and Liu X 2018 Inverse elastic scattering problems with phaseless far field data {\it arXiv:1812.02359v1}

\bibitem{JLZ18a} Ji X, Liu X and Zhang B 2019 Target reconstruction with a reference point scatterer using phaseless far field patterns {\it SIAM J. Imaging Sci.} 12(1), 372--391

\bibitem{JLZ18b} Ji X, Liu X and Zhang B 2018 Phaseless inverse source scattering problem: phase retrieval, uniqueness and direct sampling methods {\it arXiv:1808.02385v1}

%\bibitem{KarageorghisAPNUM}  Karageorghis A,  Johansson B T and Lesnic D 2012 The method of fundamental solutions for the identification of a sound-soft obstacle in inverse acoustic scattering
%{\it Applied Numerical Mathematics} {\bf 62} 1767--1780

\bibitem{KG08} Kirsch A and Grinberg N 2008 {\it The Factorization Methods for Inverse Problems}. (Oxford University Press).

\bibitem{Kli14} Klibanov M V 2014 Phaseless inverse scattering problems in three dimensions {\it SIAM J. Appl. Math.} {\bf 74} 392--410

\bibitem{Kli17} Klibanov M V 2017 A phaseless inverse scattering problem for the 3-D Helmholtz equation {\it Inverse Probl. Imaging} {\bf 11} 263--276

\bibitem{KR16} Klibanov M V and Romanov V G 2016 Reconstruction procedures for two inverse scattering problems without the phase information {\it SIAM J. Appl. Math.} {\bf 76} 178--196

\bibitem{KR17} Klibanov M V and Romanov V G 2017 Uniqueness of a 3-D coefficient inverse scattering problem without the phase information {\it Inverse Problems} {\bf 33} 095007

\bibitem{KR97} Kress R and Rundell W 1997 Inverse obstacle scattering with modulus of the far field pattern as data {\it Inverse Problems in Medical Imaging and Nondestructive Testing (Oberwolfach, 1996)} 75--92

\bibitem{Lee16} Lee K M 2016 Shape reconstructions from phaseless data, {\it Eng. Anal. Bound. Elem.}  {\bf 71} 174--178

\bibitem{LL15} Li J and Liu H 2015 Recovering a polyhedral obstacle by a few backscattering measurements {\it J. Differential Equat.}  {\bf 259} 2101--2120

\bibitem{LLW17} Li J, Liu H and Wang Y 2017 Recovering an electromagnetic obstacle by a few phaseless backscattering measurements {\it Inverse Problems} {\bf 33} 035001

\bibitem{LLZ09} Li J, Liu H and Zou J 2009 Strengthened linear sampling method with a reference ball
{\it SIAM J. Sci. Comput.} {\bf 31}(6) 4013--4040

\bibitem{Lines} Lines C 2003 {\it Inverse Scattering by Unbounded Obstacles} (PhD thesis: Brunel University)

\bibitem{LS04} Liu J and Seo J 2004 On stability for a translated obstacle with impedance boundary condition
{\it Nonlinear Anal.} {\bf 59}  731--744

\bibitem{LZ09} Liu X and Zhang B 2009 Unique determination of a sound soft ball by the modulus of a single far field datum {\it J. Math. Anal. Appl.} {\bf 365} 619--624

\bibitem{Maleki} Maleki M H, Devaney A J and Schatzberg A 1992 Tomographic reconstruction from optical scattered intensities {\it J. Opt. Soc. Am. A}  {\bf 9} 1356--1363

\bibitem{Maleki1} Maleki M H and Devaney A J 1993 Phase-retrieval and intensity-only reconstruction algorithms for optical diffraction tomography {\it J. Opt. Soc. Am. A}  {\bf 10} 1086--1092

\bibitem{Maretzke} Maretzke S and Hohage T 2017 Stability estimates for linearized near-field phase retrieval in X-ray phase contrast imaging {\it SIAM J. Appl. Math.} {\bf 77} 384--408

\bibitem{McLean} McLean W 2000 {\it Strongly Elliptic Systems and Boundary Integral Equations} (Cambridge: Cambridge University)

\bibitem{Novikov15} Novikov R G 2015 Formulas for phase recovering from phaseless scattering data at fixed frequency {\it Bull. Sci. Math.} {\bf 139 } 923--936

\bibitem{Novikov16} Novikov R G 2016 Explicit formulas and global uniqueness for phaseless inverse scattering in multidimensions {\it J. Geom. Anal.} {\bf 26 } 346--359

\bibitem{Pan} Pan L, Zhong Y, Chen X and Yeo S P 2011 Subspace-based optimization method for inverse scattering problems utilizing phaseless data {\it IEEE Trans. Geosci. Remote Sensing }  {\bf 49}  981--987

%\bibitem{Pot01} Potthast R 2001 {\it Point Sources and Multipoles in Inverse Scattering Theory} (London: Chapman \& Hall)

\bibitem{Qu} Qu F, Zhang B and Zhang H 2019 A novel integral equation for scattering by locally rough surfaces and application to the inverse problem: the Neumann case {\it arXiv:1901.08703v1}

%\bibitem{Shin} Shin J 2016 Inverse obstacle backscattering problems with phaseless data {\it Euro. J. Appl. Math.} {\bf 27} 111--130

\bibitem{SZG18} Sun F, Zhang D and Guo Y 2018 Uniqueness in phaseless inverse scattering problems with superposition of incident point sources {\it arXiv:1812.03291v1}

\bibitem{Takenaka} Takenaka T, Wall D J N, Harada H and Tanaka M 1997 Reconstruction algorithm of the refractive index of a cylindrical object from the intensity measurements of the total field  {\it Microwave Opt. Tech. Lett.}  {\bf 14}  139--197

\bibitem{Willers} Willers A 1987 The Helmholtz equation in disturbed half-spaces {\it Math. Methods Appl. Sci.} {\bf 9} 312--323

\bibitem{XZZ18a} Xu X, Zhang B and Zhang H 2018 Uniqueness in inverse scattering problems with phaseless far-field data at a fixed frequency {\it SIAM J. Appl. Math.} {\bf 78} 1737--1753

\bibitem{XZZ18b} Xu X, Zhang B and Zhang H 2018 Uniqueness in inverse scattering problems with phaseless far-field data at a fixed frequency II {\it SIAM J. Appl. Math.} {\bf 78} 3024--3039.

\bibitem{ZHW2013} Zhang H and Zhang B 2013 A novel integral equation for scattering by locally rough surfaces and application to the inverse problem {\it SIAM J. Appl. Math.} {\bf 73}  1811--1829

\bibitem{ZZ17a} Zhang B and Zhang H 2017 Recovering scattering obstacles by multi-frequency phaseless far-field data {\it Journal of Computational Physics} {\bf 345} 58--73

\bibitem{ZZ17b} Zhang B and Zhang H 2017 Imaging of locally rough surfaces from intensity-only far-field or near-field data {\it Inverse Problems} {\bf 33} 055001

\bibitem{ZZ18} Zhang B and Zhang H 2018 Fast imaging of scattering obstacles from phaseless far-field measurements at a fixed frequency {\it Inverse Problems} {\bf 34} 104005

\bibitem{ZG18}  Zhang D and  Guo Y 2018 Uniqueness results on phaseless inverse scattering with a reference ball {\it Inverse Problems} {\bf 34} 085002

\bibitem{ZGLL18} Zhang D,  Guo Y, Li J and Liu H 2018  Retrieval of acoustic sources from multi-frequency phaseless data {\it Inverse Problems} {\bf 34} 094001

\bibitem{ZhangDeyue2008} Zhang D, Ma F and Fang M 2008 A finite element method with perfectly matched absorbing layers for the wave scattering from a cavity {\it Chinese Journal of Computational Physics} {\bf 25}(3) 301--308
\end{thebibliography}
\end{document}